\documentclass[bjps]{imsart}

\usepackage{amsfonts,amsmath,latexsym,amssymb,mathrsfs,amsthm,xcolor, url}

\usepackage{float}
\RequirePackage{enumitem}
\startlocaldefs
\def\numberlikeadb{\global\def\theequation{\thesection.\arabic{equation}}}
\numberlikeadb
\newtheorem{theorem}{Theorem}[section]
\newtheorem{lemma}[theorem]{Lemma}
\newtheorem{corollary}[theorem]{Corollary}

\newtheorem{remark}[theorem]{Remark}

\newcommand{\Prob}{\mathbb{P}\,}
\newcommand{\EE}{\mathbb{E}}
\numberwithin{equation}{section}

\endlocaldefs

\begin{document}

\begin{frontmatter}

% "Title of the Paper"
\title{Multivariate normal approximation of the maximum likelihood estimator via the delta method}

\runtitle{Multivariate normal approximation of the MLE}

\begin{aug}
% indicate corresponding author with \corref{}

%\author{\fnms{Robert E} \snm{Gaunt}\thanksref{a,b}\ead[label=e1]{???}}

\author{\fnms{Andreas} \snm{Anastasiou}\thanksref{c}\ead[label=e2]{A.Anastasiou@lse.ac.uk}}

\and

 \author{\fnms{Robert E.} \snm{Gaunt}\thanksref{a}\corref{Robert E. Gaunt}\ead[label=e1]{robert.gaunt@manchester.ac.uk}\ead[label=e2,url]{www.foo.com}}

\affiliation[c]{The London School of Economics and Political Science}
\affiliation[a]{The University of Manchester}

%\address[a]{\printead{e1}}
%\address[c]{\printead{e2}}

\runauthor{A. Anastasiou and R. E. Gaunt}

\end{aug}

\begin{abstract}We use the delta method and Stein's method to derive, under regularity conditions, explicit upper bounds for the distributional distance between the distribution of the maximum likelihood estimator (MLE) of a $d$-dimensional parameter and its asymptotic multivariate normal distribution. Our bounds apply in situations in which the MLE can be written as a function of a sum of i.i.d$.$ $t$-dimensional random vectors.  We apply our general bound to establish a bound for the multivariate normal approximation of the MLE of the normal distribution with unknown mean and variance.
\end{abstract}



\end{frontmatter}

\section{Introduction} 

The asymptotic normality of maximum likelihood estimators (MLEs), under regularity conditions, is one of the most well-known and fundamental results in mathematical statistics.  In a very recent paper, \cite{a15} obtained explicit upper bounds on the distributional distance between the distribution of the MLE of a $d$-dimensional parameter and its asymptotic multivariate normal distribution.  In the present paper, we use a different approach to derive such bounds in the special situation where the MLE can be written as a function of a sum of i.i.d$.$ $t$-dimensional random vectors.  In this setting, our bounds improve on (or are at least as good as) those of \cite{a15} in terms of sharpness, simplicity and strength of distributional distance.  The process used to derive our bounds consists a combination of Stein's method and the multivariate delta method. It is a natural generalisation of the approach used by \cite{al15} to obtain improved  bounds -- only in cases where the MLE can be expressed as a function of a sum of i.i.d$.$ random variables -- than those obtained by \cite{ar15} for the asymptotic normality of the single parameter MLE.

Let us now introduce the notation and setting of this paper.  Let $\boldsymbol{\theta}=(\theta_1,\ldots,\theta_d)$ denote the unknown parameter found in a statistical model.  Let $\boldsymbol{\theta_0}=(\theta_{0_1},\ldots,\theta_{0_d})$ be the true (still unknown) value and let $\boldsymbol{\Theta}\subset\mathbb{R}^d$ denote the parameter space.  Suppose $\mathbf{X}=(\mathbf{X}_1,\ldots,\mathbf{X}_n)$ is a random sample of $n$ i.i.d$.$ $t$-dimensional random vectors with joint probability density (or mass) function $f(\boldsymbol{x}|\boldsymbol{\theta})$, where $\boldsymbol{x}=(\boldsymbol{x}_1,\boldsymbol{x}_2,\ldots,\boldsymbol{x}_n)$.  The likelihood function is $L(\boldsymbol{\theta}; \boldsymbol{x}) = f(\boldsymbol{x}|\boldsymbol{\theta})$, and its natural logarithm, called the log-likelihood function, is denoted by $l(\boldsymbol{\theta};\boldsymbol{x})$. We shall assume that the MLE $\boldsymbol{\hat{\theta}_n(X)}$ of the parameter of interest $\boldsymbol{\theta}$ exists and is unique.  It should, however, be noted that uniqueness and existence of the MLE can not be taken for granted; see, for example, \cite{Billingsley} for an example of non-uniqueness.  Assumptions that ensure existence and uniqueness of the MLE are given in \cite{Makelainen}.

In addition to assuming the existence and uniqueness of the MLE, we shall assume that the MLE takes the following form.   Let $q: \boldsymbol{\Theta} \rightarrow \mathbb{R}^d$ be such as all the entries of the function $q$ have continuous partial derivatives with respect to $\boldsymbol{\theta}$ and that
\begin{equation}
\label{qproperty}
q(\boldsymbol{\hat{\theta}_n(x)}) = \frac{1}{n}\sum_{i=1}^{n}g(\boldsymbol{x_i}) = \frac{1}{n}\sum_{i=1}^{n}\big(
g_{1}(\boldsymbol{x_i}),\ldots,g_{d}(\boldsymbol{x_i})\big)^\intercal
\end{equation}
for some $g:\mathbb{R}^t\rightarrow\mathbb{R}^d$.  This setting is a natural generalisation of that of \cite{al15} to multiparameter MLEs.  The representation (\ref{qproperty}) for the MLE allows us to approximate the MLE using a different approach to that given in \cite{a15}, which results in a bound that improves on that obtained by \cite{a15}.  As we shall see in Section 3, a simple and important example of an MLE of the form (\ref{qproperty}) is given by the normal distribution with unknown mean and variance:
\begin{equation*}f(x|\boldsymbol{\theta}) = \frac{1}{\sqrt{2\pi\sigma^2}}\exp\left\lbrace -\frac{1}{2\sigma^2}(x-\mu)^2 \right\rbrace,\quad x \in \mathbb{R}. 
\end{equation*}

% $f(x|\mu,\sigma^2)=\aarg{(2\pi\sigma^2)^{-1/2}\exp\big(-(x-\mu)^2/(2\sigma^2)\big)}$, $x\in\mathbb{R}$.  

In this paper, we derive general bounds for the distributional distance between an MLE of the form (\ref{qproperty}) and its limiting multivariate normal distribution. The rest of the paper is organised as follows. In Section 2, we derive a general upper bound on the distributional distance between the distribution of the MLE of a $d$-dimensional parameter and the limiting multivariate normal distribution for situations in which (\ref{qproperty}) is satisfied.  In Section 3, we apply our general bound to obtain a bound for the multivariate normal approximation of the MLE of the normal distribution with unknown mean and variance.

\section{The general bound}
\label{sec:general_bound}

In this section, we give quantitative results in order to compliment the qualitative result for the multivariate normal approximation of the MLE as expressed below in Theorem \ref{Theorem_asymptotic_MLE_delta}. In the statement of this theorem, and the rest of the paper, $I(\boldsymbol{\theta_0})$ denotes the expected Fisher information matrix for one random vector, while $[I(\boldsymbol{\theta_0})]^{\frac{1}{2}}$ denotes the principal square root of $I(\boldsymbol{\theta_0})$.  Also, $I_{d \times d}$ denotes the $d\times d$ identity matrix.  To avoid confusion with the expected Fisher information matrix, the subscript will always be included in the notation for the $d\times d$ identity matrix.  Following \cite{Davison}, the following assumptions are made in order for the asymptotic normality of the vector MLE to hold:
\begin{itemize}[leftmargin=0.55in]
\item[(R.C.1)] The densities defined by any two different values of $\boldsymbol{\theta}$ are distinct.
\item[(R.C.2)] $\ell(\boldsymbol{\theta};\boldsymbol{x})$ is three times differentiable with respect to the unknown vector parameter, $\boldsymbol{\theta}$, and the third partial derivatives are continuous in $\boldsymbol{\theta}$.
\item[(R.C.3)] For any $\boldsymbol{\theta_0} \in \boldsymbol{\Theta}$ and for $\boldsymbol{\mathbb{X}}$ denoting the support of the data, there exists $\epsilon_0 > 0$ and functions $M_{rst}(\boldsymbol{x})$ (they can depend on $\boldsymbol{\theta_0}$), such that for $\boldsymbol{\theta} = (\theta_1, \theta_2, \ldots, \theta_d)$ and $r, s, t, j = 1,2,\ldots,d,$
\begin{equation}
\nonumber \left|\frac{\partial^3}{\partial \theta_r \partial \theta_s \partial \theta_t}\ell(\boldsymbol{\theta};\boldsymbol{x})\right| \leq M_{rst}(\boldsymbol{x}), \; \forall \boldsymbol{x} \in \boldsymbol{\mathbb{X}},\; \left|\theta_j - \theta_{0,j}\right| < \epsilon_0,
\end{equation}
with $\EE[M_{rst}(\boldsymbol{X})] < \infty$.
\item[(R.C.4)] For all $\boldsymbol{\theta} \in \Theta$, $\EE_{\boldsymbol{\theta}}[\ell_{\boldsymbol{X_i}}(\boldsymbol{\theta})] = 0$.
\item[(R.C.5)] The expected Fisher information matrix for one random vector $I(\boldsymbol{\theta})$ is finite, symmetric and positive definite. For $r,s=1,2,\ldots,d$, its elements satisfy
\begin{equation}
\nonumber n[I(\boldsymbol{\theta})]_{rs} = \EE\bigg[ \frac{\partial}{\partial \theta_r} \ell(\boldsymbol{\theta};\boldsymbol{X})\frac{\partial}{\partial \theta_s}\ell(\boldsymbol{\theta};\boldsymbol{X})\bigg] = \EE\bigg[-\frac{\partial^2}{\partial \theta_r \partial \theta_s} \ell(\boldsymbol{\theta};\boldsymbol{X}) \bigg].
\end{equation}
This  implies that $nI(\boldsymbol{\theta})$ is the covariance matrix of $\nabla(\ell(\boldsymbol{\theta};\boldsymbol{x}))$.
\end{itemize}
These regularity conditions in the multi-parameter case resemble those in \cite{ar15} where the parameter is scalar. The following theorem gives the qualitative result for the asymptotic normality of a vector MLE (see \cite{Davison} for a proof).
\begin{theorem}
\label{Theorem_asymptotic_MLE_delta}
Let $\boldsymbol{X_1}, \boldsymbol{X_2}, \ldots, \boldsymbol{X_n}$ be i.i.d$.$ random vectors with probability density (mass) functions $f(\boldsymbol{x_i}|\boldsymbol{\theta})$, $1\leq i\leq n$, where $\theta \in \Theta \subset \mathbb{R}^d$.  Also let $\boldsymbol{Z} \sim {\rm N}_{d}\left(\boldsymbol{0},I_{d\times d}\right)$, where $\boldsymbol{0}$ is the $d \times 1$ zero vector and $I_{d \times d}$ is the $d\times d$ identity matrix. Then, under (R.C.1)-(R.C.5),
\begin{equation}
\label{scoredistribution_multi}
\nonumber \sqrt{n}[I(\boldsymbol{\theta_0})]^{\frac{1}{2}}(\boldsymbol{\hat{\theta}_n(X)} - \boldsymbol{\theta_0}) \xrightarrow[{n \to \infty}]{{\rm d}} {\rm N}_d(\boldsymbol{0}, I_{d\times d}).
\end{equation}
\end{theorem}

Let us now introduce some notation. We let $C_b^n(\mathbb{R}^d)$ denote the space of bounded functions on $\mathbb{R}^d$ with bounded $k$-th order derivatives for $k\leq n$. We abbreviate $|h|_1 := \sup_i\big\|\frac{\partial}{\partial x_i}h\big\|$ and $|h|_2 := \sup_{i,j}\big\|\frac{\partial^2}{\partial x_i\partial x_j}h\big\|$, where $\|\cdot\|=\|\cdot\|_\infty$ is the supremum norm.  For ease of presentation, we let the subscript $(m)$ denote an index for which $\big|\hat{\theta}_n(\boldsymbol{x})_{(m)} - \theta_{0_{(m)}}\big|$ is the largest among the $d$ components:
\begin{equation}
\nonumber (m) \in \left\lbrace 1,2,\ldots,d\right\rbrace : \big|\hat{\theta}_n(\boldsymbol{x})_{(m)} - \theta_{0_{(m)}}\big| \geq \big|\hat{\theta}_n(\boldsymbol{x})_j - \theta_{0_j}\big|, \;\; \forall j \in \left\lbrace 1,2,\ldots, d \right\rbrace,
\end{equation}
and also
\begin{equation}
\label{cm}
Q_{(m)}=Q_{(m)}(\boldsymbol{X},\boldsymbol{\theta_0}) := \hat{\theta}_n(\boldsymbol{X})_{(m)} - \theta_{0_{(m)}}.
\end{equation}
The following theorem gives the main result of this paper.
\begin{theorem}
\label{Theorem_delta_multi}
Let $\boldsymbol{X_1}, \boldsymbol{X_2}, \ldots, \boldsymbol{X_n}$ be i.i.d$.$ $\mathbb{R}^t$-valued random vectors with probability density (or mass) function $f(\boldsymbol{x_i}|\boldsymbol{\theta})$, for which the parameter space $\Theta$ is an open subset of $\mathbb{R}^d$.  Let $\boldsymbol{Z}\sim {\mathrm {N}}_d(\boldsymbol{0},I_{d \times d})$.  Assume that the MLE $\boldsymbol{\hat{\theta}_n(X)}$ exists and is unique and that (R.C.1)-(R.C.5) are satisfied. Let $q:\Theta \rightarrow \mathbb{R}^d$ be twice differentiable and $g:\mathbb{R}^t\rightarrow\mathbb{R}^d$ such that the special structure \eqref{qproperty} is satisfied. Let
\begin{equation*}
\xi_{ij} = \sum_{l=1}^{d}\big[K^{-\frac{1}{2}}\big]_{kl}\left(g_l(\boldsymbol{X_i}) - q_l(\boldsymbol{\theta_0})\right),
\end{equation*}
where
\begin{equation}
\label{Kmatrix_general}
K= \left[\nabla q(\boldsymbol{\theta_0})\right]\left[I(\boldsymbol{\theta_0})\right]^{-1}\left[\nabla q(\boldsymbol{\theta_0})\right]^{\intercal}.
\end{equation}
Also, let $Q_{(m)}$ be as in \eqref{cm}.  Then, for all $h\in C_b^2(\mathbb{R}^d)$,
\begin{align}
\label{bounddeltageneral}
\nonumber & \left|\mathbb{E}\left[h\left(\sqrt{n}\left[I(\boldsymbol{\theta_0})\right]^{\frac{1}{2}}\left(\boldsymbol{\hat{\theta}_n(X)} - \boldsymbol{\theta_0}\right)\right)\right]- \mathbb{E}[h(\boldsymbol{Z})]\right|\\
\nonumber & \leq \frac{\sqrt{\pi}|h|_2}{4\sqrt{2}n^{3/2}}\sum_{i=1}^n\sum_{j,k,l=1}^d\Big\{\mathbb{E}\left|\xi_{ij}\xi_{ik}\xi_{il}\right|+2\left|\mathbb{E}\left[\xi_{ij}\xi_{ik}\right]\right|\mathbb{E}\left|\xi_{il}\right|\Big\} \nonumber\\
&\quad+ 2\frac{\|h\|}{\epsilon^2}\mathbb{E}\left[\sum_{j=1}^{d}\left(\hat{\theta}_n(\boldsymbol{X})_j-\theta_{0_j}\right)^2\right]\nonumber\\
& \quad+ \frac{|h|_1}{\sqrt{n}}\sum_{j=1}^{d}\!\Big\{\mathbb{E}\big[\big| A_{1j}(\boldsymbol{\theta_0},\boldsymbol{X})\big|\,\big|\,\left|Q_{(m)}\right|< \epsilon\big] \!+\!\mathbb{E}\big[\big| A_{2j}(\boldsymbol{\theta_0},\boldsymbol{X})\big|\,\big|\,\left|Q_{(m)}\right|< \epsilon\big]\Big\},
\end{align}
where
\begin{align*}
%\label{Anotation}
\nonumber  A_{1j}(\boldsymbol{\theta_0},\boldsymbol{X}) &= n\sum_{k=1}^d\!\left\lbrace\!\left(\left[I(\boldsymbol{\theta_0})\right]^{\frac{1}{2}}\left[\nabla q(\boldsymbol{\theta_0})\right]^{-1} - K^{-\frac{1}{2}}\right)_{jk}\!\left(q_k\big(\boldsymbol{\hat{\theta}_n(X)}\big)-q_k(\boldsymbol{\theta_0})\right)\!\right\rbrace\\
\nonumber  A_{2j}(\boldsymbol{\theta_0},\boldsymbol{X}) &= \frac{n}{2}\sum_{k=1}^{d}\!\bigg\{\!\!\vphantom{(\left(\sup_{\theta:|\theta-\theta_0|\leq\epsilon}\left|l^{(3)}(\theta;\boldsymbol{X})\right|\right)^2}\left(\left[I(\boldsymbol{\theta_0})\right]^{\frac{1}{2}}\left[\nabla q(\boldsymbol{\theta_0})\right]^{-1}\right)_{jk}\sum_{m=1}^{d}\sum_{l=1}^{d}\!\big(\hat{\theta}_n(\boldsymbol{X})_m - \theta_{0_m}\big)\\
&\quad\times\big(\hat{\theta}_n(\boldsymbol{X})_l - \theta_{0_l}\big) M_{klm}(\boldsymbol{X})\bigg\}.
\end{align*}
\end{theorem}

\begin{remark}
\label{remarkdeltamulti} 
\noindent{1.} For fixed $d$, it is clear that the first term is $O(n^{-1/2})$.  For the second term, using \cite{cs68} we know that the mean squared error (MSE) of the MLE is of order $n^{-1}$. This yields, for fixed $d$, $\mathbb{E}\big[\sum_{j=1}^{d}(\hat{\theta}_n(\boldsymbol{X})_j - \theta_{0,j})^2\big] = O(n^{-1})$.  To establish that (\ref{bounddeltageneral}) is order $n^{-1/2}$, it therefore suffices to show that for the third term the quantities $\mathbb{E}\big[\big| A_{1j}(\boldsymbol{\theta_0},\boldsymbol{X})\big|\,\big|\,\left|Q_{(m)}\right|< \epsilon\big]$ and $\mathbb{E}\big[\big| A_{2j}(\boldsymbol{\theta_0},\boldsymbol{X})\big|\,\big|\,\left|Q_{(m)}\right|< \epsilon\big]$ are $O(1)$.  For $A_{2j}(\boldsymbol{\theta_0},\boldsymbol{X})$, we use the Cauchy-Schwarz inequality to get that
\begin{align}\label{CSineq}&\mathbb{E}|\big(\hat{\theta}_n(\boldsymbol{X})_m - \theta_{0_m}\big)\big(\hat{\theta}_n(\boldsymbol{X})_l - \theta_{0_l}\big) M_{jkl}(\boldsymbol{X})|\nonumber \\
&\leq \Big[\mathbb{E}\big[\big(\hat{\theta}_n(\boldsymbol{X})_m - \theta_{0_m}\big)^2\big(\hat{\theta}_n(\boldsymbol{X})_l - \theta_{0_l}\big)^2\big]\Big]^{1/2} \nonumber\\
&\quad\times\Big[\mathbb{E}\big[(M_{klm}(\boldsymbol{X}))^2\,\big|\,|Q_{(m)}|<\epsilon\big]\Big]^{1/2}.
\end{align}
%\begin{equation}
%\label{CSineq}
%\mathbb{E}\left[\big(\hat{\theta}_n(\boldsymbol{X})_m - \theta_{0_m}\big)\big(\hat{\theta}_n(\boldsymbol{X})_l - \theta_{0_l}\big)\right] \leq \sqrt{\mathbb{E}\left[\big(\hat{\theta}_n(\boldsymbol{X})_m - \theta_{0_m}\big)^2\right]}\sqrt{\mathbb{E}\left[\big(\hat{\theta}_n(\boldsymbol{X})_l - \theta_{0_l}\big)^2\right]}.
%\end{equation}
Since from \cite{cs68} the MSE is of order $n^{-1}$, we apply \eqref{CSineq} to the expression of $A_{2j}(\boldsymbol{\theta_0},\boldsymbol{X})$ to get that $\mathbb{E}\big[\big| A_{2j}(\boldsymbol{\theta_0},\boldsymbol{X})\big|\,\big|\,\left|Q_{(m)}\right|< \epsilon\big] = O(1)$.  Whilst we suspect that $\mathbb{E}\big[\big| A_{1j}(\boldsymbol{\theta_0},\boldsymbol{X})\big|\,\big|\,\left|Q_{(m)}\right|< \epsilon\big]$ are generally $O(1)$, we have been unable to establish this fact.  In the univariate $d=1$ case, this term vanishes (see \cite{al15}), and this is also the case in our application in Section 3 (in which case the matrices $I(\boldsymbol{\theta_0})$ and $\nabla q(\boldsymbol{\theta_0})$ are diagonal).

\vspace{3mm}

\noindent{2.} An upper bound on the quantity $|\mathbb{E}[h(\sqrt{n}[I(\boldsymbol{\theta_0})]^{\frac{1}{2}}(\boldsymbol{\hat{\theta}_n(X)} - \boldsymbol{\theta_0}))]-\mathbb{E}[h(\mathbf{Z})]|$ was given in \cite{a15}. A brief comparison of that bound and our bound follows. The bound in \cite{a15} is more general than our bound. It covers the case of independent but not necessarily identically distributed random vectors. In addition, the bound in \cite{a15} is more general in the sense that it is applicable whatever the form of the MLE is; our bound can be used only when \eqref{qproperty} is satisfied. Furthermore, \cite{a15} gives bounds even for cases where the MLE cannot be expressed analytically. On the other hand, our bound is preferable when the MLE has a representation of the form (\ref{qproperty}). Our bound in \eqref{bounddeltageneral} applies to a wider class of test functions $h$; our class is $C_b^2(\mathbb{R}^d)$, as opposed to $C_b^3(\mathbb{R}^d)$, and has a  better dependence on the dimension $d$; for fixed $n$, our bound is $O(d^4)$, whereas the bound of \cite{a15} is $O(d^6)$. Moreover, in situations where the MLE is already a sum of independent terms our bound is easier to use since if we apply $q(\boldsymbol{x}) = \boldsymbol{x}$ to \eqref{remaindertermsR1R2} yields $R_2 = 0$ in \eqref{trianglefirststep}. This means that our bound \eqref{bounddeltageneral} simplifies to
\begin{align*}
&\left|\mathbb{E}\left[h\left(\sqrt{n}\left[I(\boldsymbol{\theta_0})\right]^{\frac{1}{2}}\left(\boldsymbol{\hat{\theta}_n(X)} - \boldsymbol{\theta_0}\right)\right)\right]- \mathbb{E}[h(\boldsymbol{Z})]\right| \\
&\quad \leq \frac{\sqrt{\pi}|h|_2}{4\sqrt{2}n^{3/2}}\sum_{i=1}^n\sum_{j,k,l=1}^d\Big\{\mathbb{E}\left|\xi_{ij}\xi_{ik}\xi_{il}\right|+2\left|\mathbb{E}\left[\xi_{ij}\xi_{ik}\right]\right|\mathbb{E}\left|\xi_{il}\right|\Big\}.
\end{align*}
Another simplification of the general bound is given in Section 3, where $X_1,X_2,\ldots,X_n$ are taken to be i.i.d$.$ normal random variables with unknown mean and variance, which can also be treated by \cite{a15}. The simpler terms in our bound lead to a much improved bound for this example (see Remark \ref{remark3.2} for more details).

\vspace{3mm}

\noindent{3.} If for all $k,m,l \in \left\lbrace 1,2,\ldots,d \right\rbrace$, $\big|\frac{\partial^2}{\partial\theta_m\partial\theta_l}q_k(\boldsymbol{\theta})\big|$ is uniformly bounded in $\Theta$, then we do not need to use $\epsilon$ or conditional expectations in deriving an upper bound (we essentially apply Theorem \ref{Theorem_delta_multi} with $\epsilon\rightarrow\infty$). This leads to the following simpler bound:
\begin{align}
\label{uniformlyboundedq}
\nonumber & \left|\mathbb{E}\left[h\left(\sqrt{n}\left[I(\boldsymbol{\theta_0})\right]^{\frac{1}{2}}\left(\boldsymbol{\hat{\theta}_n(X)} - \boldsymbol{\theta_0}\right)\right)\right]- \mathbb{E}[h(\boldsymbol{Z})]\right|\\
 &\quad \leq \frac{\sqrt{\pi}|h|_2}{4\sqrt{2}n^{3/2}}\sum_{i=1}^n\sum_{j,k,l=1}^d\Big\{\mathbb{E}|\xi_{ij}\xi_{ik}\xi_{il}|+2|\mathbb{E}[\xi_{ij}\xi_{ik}]|\mathbb{E}|\xi_{il}|\Big\} \nonumber\\
 &\quad\quad+\frac{|h|_1}{\sqrt{n}}\sum_{j=1}^{d}\Big\{\mathbb{E}\big| A_{1j}(\boldsymbol{\theta_0},\boldsymbol{X})\big| +\mathbb{E}\big| A_{2j}(\boldsymbol{\theta_0},\boldsymbol{X})\big|\Big\}.
 \end{align}
\end{remark}

To prove Theorem \ref{Theorem_delta_multi}, we employ Stein's method and the multivariate delta method. Our strategy consists of benefiting from the special form of $q\big(\boldsymbol{\hat{\theta}_n(X)}\big)$, which is a sum of random vectors. It is here where multivariate delta method comes into play; instead of comparing $\boldsymbol{\hat{\theta}_n(X)}$ to $\boldsymbol{Z}\sim {\mathrm {N}}_d(\boldsymbol{0},I_{d \times d})$, we compare $q\big(\boldsymbol{\hat{\theta}_n(X)}\big)$ to $\boldsymbol{Z}$ and then find upper bounds on the distributional distance between the distribution of $\boldsymbol{\hat{\theta}_n(X)}$ and $q\big(\boldsymbol{\hat{\theta}_n(X)}\big)$. As well as recalling the multivariate delta method, we state and prove two lemmas that will be used in the proof of Theorem \ref{Theorem_delta_multi}.  

\begin{theorem}\label{thm2.2} \textbf{(Multivariate delta method)}
Let the parameter $\boldsymbol{\theta} \in \Theta$, where $\Theta \subset \mathbb{R}^d$. Furthermore, let $Y_n$ be a sequence of random vectors in $\mathbb{R}^d$ which, for $\boldsymbol{Z} \sim {\mathrm {N}}_d(\boldsymbol{0},I_{d \times d})$, satisfies
\begin{equation}
\nonumber \sqrt{n}\left(Y_n - \boldsymbol{\theta}\right) \xrightarrow[{n \to \infty}]{{\rm d}} \Sigma^{\frac{1}{2}}\boldsymbol{Z},
\end{equation}
where $\Sigma$ is a symmetric, positive definite covariance matrix. Then, for any function $b:\mathbb{R}^d \rightarrow \mathbb{R}^d$ of which the Jacobian matrix $\nabla b(\boldsymbol{\theta}) \in \mathbb{R}^{d \times d}$ exists and is invertible, we have that for $\Lambda := \left[\nabla b(\boldsymbol{\theta})\right]\Sigma\left[\nabla b(\boldsymbol{\theta})\right]^{\intercal}$,
\begin{equation*}
\sqrt{n}\left(b\left(Y_n\right) - b(\boldsymbol{\theta})\right) \xrightarrow[{n \to \infty}]{{\rm d}} \Lambda^{\frac{1}{2}}\boldsymbol{Z}.
\end{equation*}
\end{theorem}

\begin{lemma}\label{lem2.3}Let
\begin{align}
\label{W}
\boldsymbol{W} = \frac{K^{-\frac{1}{2}}}{\sqrt{n}}\sum_{i=1}^{n}\left[g(\boldsymbol{X_i}) - q(\boldsymbol{\theta_0})\right],
\end{align}
where $K$ is as in \eqref{Kmatrix_general}. Then $\mathbb{E}\mathbf{W}=\mathbf{0}$ and $\mathrm{Cov}(\mathbf{W})=I_{d\times d}$.
\end{lemma}

\begin{proof}By the multivariate central limit theorem,
\begin{equation}\label{eqn11}\frac{1}{\sqrt{n}}\sum_{i=1}^n(g(\mathbf{X}_i)-\mathbb{E}g(\mathbf{X}_1))\xrightarrow[{n \to \infty}]{{\rm d}}\mathrm{N}_d(\mathbf{0},\mathrm{Cov}(g(\mathbf{X}_1)), 
\end{equation}
and, by Theorem \ref{thm2.2},
\begin{equation}\label{eqn22}\sqrt{n}\left(q\left(\boldsymbol{\hat{\theta}_n(X)}\right)-q(\boldsymbol{\theta_0})\right)\xrightarrow[{n \to \infty}]{{\rm d}}\mathrm{N}_d(\mathbf{0},K).
\end{equation}
Since $q(\boldsymbol{\hat{\theta}_n(X)}) = \frac{1}{n}\sum_{i=1}^ng(\mathbf{X}_i)$ and $\boldsymbol{X}_1,\boldsymbol{X}_2,\ldots,\boldsymbol{X}_n$ are i.i.d$.$ random vectors, it follows that $\mathbb{E}q(\boldsymbol{\theta_0})=\mathbb{E}g(\mathbf{X}_1)$. Furthermore, $K=\mathrm{Cov}(g(\mathbf{X}_1))$ because the limiting multivariate normal distributions in (\ref{eqn11}) and (\ref{eqn22}) must have the same mean and covariance.  It is now immediate that $\mathbb{E}\mathbf{W}=\mathbf{0}$, and a simple calculation using that $K=\mathrm{Cov}(g(\mathbf{X}_1))$ yields that $\mathrm{Cov}(\mathbf{W})=I_{d\times d}$.
\end{proof}

\begin{lemma}
\label{lemma_Stein_Delta}
Suppose that $X_{1,j},\ldots,X_{n,j}$ are independent for a fixed $j$, but that the random variables $X_{i,1},\ldots,X_{i,d}$ may be dependent for any fixed $i$.  For $j=1,\ldots,d$, let $W_j=\frac{1}{\sqrt{n}}\sum_{i=1}^nX_{ij}$ and denote $\mathbf{W}=(W_1,\ldots,W_d)^\intercal$.  Suppose that $\mathbb{E}\left[\mathbf{W}\right]=\mathbf{0}$ and $\mathrm{Cov}(\mathbf{W})=I_{d\times d}$. Then, for all $h\in C_b^2(\mathbb{R}^d)$,
\begin{align}
\label{steinbound}
\left|\mathbb{E}\left[h(\mathbf{W})\right]-\mathbb{E}\left[h(\mathbf{Z})\right]\right|&\leq\frac{\sqrt{\pi}|h|_2}{4\sqrt{2}n^{3/2}}\sum_{i=1}^n\sum_{j,k,l=1}^d\Big\{\mathbb{E}\left|X_{ij}X_{ik}X_{il}\right|\nonumber\\
&\quad+2\left|\mathbb{E}\left[X_{ij}X_{ik}\right]\right|\mathbb{E}\left|X_{il}\right|\Big\}.
\end{align}
\end{lemma}

\begin{proof} The following bound follows immediately from Lemma 2.1 of \cite{gaunt_friedman}:
\begin{align}
\label{springz}\left|\mathbb{E}\left[h(\mathbf{W})\right]-\mathbb{E}\left[h(\mathbf{Z})\right]\right|
&\leq\frac{1}{2n^{3/2}}\bigg\|\frac{\partial^3f(\mathbf{x})}{\partial x_{i_1}\partial x_{i_2}\partial x_{i_3}}\bigg\|\sum_{i=1}^n\sum_{j,k,l=1}^d\Big\{\mathbb{E}\left|X_{ij}X_{ik}X_{il}\right|\nonumber\\
&\quad+2\left|\mathbb{E}\left[X_{ij}X_{ik}\right]\right|\mathbb{E}\left|X_{il}\right|\Big\},
\end{align}
where $f$ solves the so-called Stein equation
\begin{equation}
\nonumber \nabla^\intercal\nabla f(\mathbf{w})-\mathbf{w}^\intercal\nabla f(\mathbf{w})=h(\mathbf{w})-\mathbb{E}\left[h(\mathbf{Z})\right].
\end{equation} 
From Proposition 2.1 of \cite{gaunt_rate}, we have the bound 
\[\bigg\|\frac{\partial^3f(\mathbf{x})}{\partial x_{i_1}\partial x_{i_2}\partial x_{i_3}}\bigg\|\leq\frac{\sqrt{\pi}}{2\sqrt{2}}|h|_2,\]
 which when applied to (\ref{springz}) yields (\ref{steinbound}).
\end{proof}
 
\noindent{\emph{Proof of Theorem \ref{Theorem_delta_multi}.}} The triangle inequality yields
\begin{equation}
\label{trianglefirststep}
\left|\mathbb{E}\left[ h\left(\sqrt{n}[I(\boldsymbol{\theta_0})]^{\frac{1}{2}}\left(\boldsymbol{\hat{\theta}_n(X)} - \boldsymbol{\theta_0}\right)\right)\right] - \mathbb{E}[h(\boldsymbol{Z})]\right|  \leq R_1+R_2,
\end{equation}
where
\begin{align}
\label{remaindertermsR1R2}
\nonumber R_1&\leq\left|\mathbb{E}\left[ h\left(\sqrt{n}K^{-\frac{1}{2}}\left(q\big(\boldsymbol{\hat{\theta}_n(X)}\big) - q\left(\boldsymbol{\theta_0}\right)\right) \right) \right] - \mathbb{E}[h(\boldsymbol{Z})]\right|,\\
R_2&\leq\Big|\mathbb{E}\Big[h\left(\sqrt{n}[I(\boldsymbol{\theta_0})]^{\frac{1}{2}}\left(\boldsymbol{\hat{\theta}_n(X)} -\boldsymbol{\theta_0}\right)\right) \nonumber \\
&\quad - h\left(\sqrt{n}K^{-\frac{1}{2}}\left(q\big(\boldsymbol{\hat{\theta}_n(X)}\big) -q\left(\boldsymbol{\theta_0}\right)\right)\right)\Big]\Big|.
\end{align}
The rest of the proof focuses on bounding the terms $R_1$ and $R_2$.

\vspace{3mm}

\noindent{\textbf{Step 1: Upper bound for} $R_1$}. For ease of presentation, we denote by 
\begin{align}
\nonumber \boldsymbol{W} = \sqrt{n}K^{-\frac{1}{2}}\left(q(\boldsymbol{\hat{\theta}_n(X)}) - q(\boldsymbol{\theta_0})\right) = \frac{K^{-\frac{1}{2}}}{\sqrt{n}}\sum_{i=1}^{n}\left[g(\boldsymbol{X_i}) - q(\boldsymbol{\theta_0})\right],
\end{align}
as in \eqref{W}. Thus, $\boldsymbol{W} = (W_1,W_2,\ldots,W_d)^{\intercal}$ with
\begin{equation}
\nonumber W_j = \frac{1}{\sqrt{n}}\sum_{i=1}^{n}\sum_{l=1}^{d}\big[K^{-\frac{1}{2}}\big]_{jl}\left(g_{l}(\boldsymbol{X_i})-q_{l}(\boldsymbol{\theta_0})\right) = \frac{1}{\sqrt{n}}\sum_{i=1}^{n}\xi_{ij}, \:\:\forall j \in \left\lbrace 1,\ldots,d \right\rbrace,
\end{equation}
where $\xi_{ij} = \sum_{l=1}^{d}\big[K^{-\frac{1}{2}}\big]_{jl}\left(g_l(\boldsymbol{X_i}) - q_l(\boldsymbol{\theta_0})\right)$
are independent random variables, since the $\boldsymbol{X_i}$ are assumed to be independent. Therefore, $W_j$ can be expressed as a sum of independent terms.  This together with Lemma \ref{lem2.3} ensures that the assumptions of Lemma \ref{lemma_Stein_Delta} are satisfied. Therefore, \eqref{steinbound} yields
\begin{align*}
R_1 &= \left|\mathbb{E}[h(\boldsymbol{W})] - \mathbb{E}[h(\boldsymbol{Z})]\right| \\
&\leq \frac{\sqrt{\pi}|h|_2}{4\sqrt{2}n^{3/2}}\sum_{i=1}^n\sum_{j,k,l=1}^d\Big\{\mathbb{E}|\xi_{ij}\xi_{ik}\xi_{il}|+2|\mathbb{E}[\xi_{ij}\xi_{ik}|\mathbb{E}|\xi_{il}|\Big\}.
\end{align*}
\vspace{0.05in}
\\
\textbf{Step 2: Upper bound for} $R_2$. For ease of presentation, we let
\begin{align*}
C_1:=C_1(h,q,\boldsymbol{X},\boldsymbol{\theta_0})&:=h\left(\sqrt{n}[I(\boldsymbol{\theta_0})]^{\frac{1}{2}}\left(\boldsymbol{\hat{\theta}_n(X)} -\boldsymbol{\theta_0}\right)\right)\\
&\quad - h\left(\sqrt{n}K^{-\frac{1}{2}}\left(q\big(\boldsymbol{\hat{\theta}_n(X)}\big) -q\left(\boldsymbol{\theta_0}\right)\right)\right),
\end{align*}
and our aim is to find an upper bound for $|\mathbb{E}[C_1]|$. Let us denote by $[A]_{[j]}$ the $j^{{\rm th}}$ row of a matrix $A$.  We begin by using a first order Taylor expansion to obtain
\begin{align}
\nonumber h\left(\sqrt{n}[I(\boldsymbol{\theta_0})]^{\frac{1}{2}}\left(\boldsymbol{\hat{\theta}_n(x)} -\boldsymbol{\theta_0}\right)\right)& = h\left(\sqrt{n}K^{-\frac{1}{2}}\left(q\big(\boldsymbol{\hat{\theta}_n(X)}\big) -q\left(\boldsymbol{\theta_0}\right)\right)\right)\\
\nonumber & \!\!\!\!\!\!\!\!\!\!\!\!\!\!\!\!\!\!\!\!\!\!\!\!+ \sum_{j=1}^{d}\Big\{ \sqrt{n}[[I(\boldsymbol{\theta_0})]^{\frac{1}{2}}]_{[j]}\left(\boldsymbol{\hat{\theta}_n(x)} - \boldsymbol{\theta_0}\right) \nonumber \\
&\!\!\!\!\!\!\!\!\!\!\!\!\!\!\!\!\!\!\!\!\!\!\!\! - \sqrt{n}\big[K^{-\frac{1}{2}}\big]_{[j]}\left(q\big(\boldsymbol{\hat{\theta}_n(X)}\big)-q(\boldsymbol{\theta_0})\right)\Big\}\frac{\partial}{\partial x_j}h(t(\boldsymbol{x})),\nonumber
\end{align}
where $t(\boldsymbol{x})$ is between $\sqrt{n}[I(\boldsymbol{\theta_0})]^{\frac{1}{2}}\big(\boldsymbol{\hat{\theta}_n(x)} -\boldsymbol{\theta_0}\big)$ and $\sqrt{n}K^{-\frac{1}{2}}\big(q\big(\boldsymbol{\hat{\theta}_n(X)}\big) -q\left(\boldsymbol{\theta_0}\right)\big)$. Rearranging the terms gives
\begin{align}
\label{boundC1}
|C_1|& \leq |h|_1\bigg|\sum_{j=1}^{d}\sqrt{n}\left[\left[I(\boldsymbol{\theta_0})\right]^{\frac{1}{2}}\right]_{[j]}\left(\boldsymbol{\hat{\theta}_n(x)} - \boldsymbol{\theta_0}\right)\nonumber\\
&\quad - \sqrt{n}\big[K^{-\frac{1}{2}}\big]_{[j]}\left(q\big(\boldsymbol{\hat{\theta}_n(X)}\big) - q(\boldsymbol{\theta_0})\right)\bigg|.
\end{align}
For $\boldsymbol{\theta_0^*}$ between $\boldsymbol{\hat{\theta}_n(x)}$ and $\boldsymbol{\theta_0}$, a second order Taylor expansion of $q_j\big(\boldsymbol{\hat{\theta}_n(x)}\big)$ about $\boldsymbol{\theta_0}$ gives that, for all $ j \in \left\lbrace 1,2,\ldots,d \right\rbrace$,
\begin{align}
\nonumber  q_j\big(\boldsymbol{\hat{\theta}_n(x)}\big) &= q_j(\boldsymbol{\theta_0}) + \sum_{k=1}^{d}\left(\hat{\theta}_n(\boldsymbol{x})_k - \theta_{0_k}\right)\frac{\partial}{\partial \theta_k}q_j(\boldsymbol{\theta_0})\\
\nonumber & \quad+ \frac{1}{2}\sum_{k=1}^{d}\sum_{l=1}^{d}\left(\hat{\theta}_n(\boldsymbol{x})_k - \theta_{0_k}\right)\left(\hat{\theta}_n(\boldsymbol{x})_l - \theta_{0_l}\right)\frac{\partial^2}{\partial\theta_k\partial\theta_l}q_j\left(\boldsymbol{\theta_0^*}\right)
\end{align}
and after a small rearrangement of the terms we get that
\begin{align*}
\nonumber  [\nabla q(\boldsymbol{\theta_0})]_{[j]}\big(\boldsymbol{\hat{\theta}_n(x)} - \boldsymbol{\theta_0}\big)&= q_j\big(\boldsymbol{\hat{\theta}_n(x)}\big) - q_j(\boldsymbol{\theta_0})\\
&\quad\!\!\!\!\!\!\!\!\!\!\!\!\!\!\!\!\!\!\!\!\!\!\!\! - \frac{1}{2}\sum_{k=1}^{d}\sum_{l=1}^{d}\big(\hat{\theta}_n(\boldsymbol{x})_k - \theta_{0_k}\big)\big(\hat{\theta}_n(\boldsymbol{x})_l - \theta_{0_l}\big)\frac{\partial^2}{\partial\theta_k\partial\theta_l}q_j(\boldsymbol{\theta_0^*}).
\end{align*}
Since this holds for all $ j\in\left\lbrace 1,2,\ldots,d \right\rbrace$, we have that
\begin{align*}
\nonumber \nabla q(\boldsymbol{\theta_0})\big(\boldsymbol{\hat{\theta}_n(x)} - \boldsymbol{\theta_0}\big) &= q\big(\boldsymbol{\hat{\theta}_n(X)}\big) - q(\boldsymbol{\theta_0}) \\
&\quad\!\!\!\!\!\!\!\!\!\!\!\!\!\!\!\!\!\!\!\!\!\!\!\!- \frac{1}{2}\sum_{k=1}^{d}\sum_{l=1}^{d}\big(\hat{\theta}_n(\boldsymbol{x})_k - \theta_{0_k}\big)\big(\hat{\theta}_n(\boldsymbol{x})_l - \theta_{0_l}\big)\frac{\partial^2}{\partial\theta_k\partial\theta_l}q(\boldsymbol{\theta_0^*}),
\end{align*}
which then leads to
\begin{align}
\label{mid_step_proof}
\nonumber & \sqrt{n}[I(\boldsymbol{\theta_0})]^{\frac{1}{2}}\left(\boldsymbol{\hat{\theta}_n(X)} - \boldsymbol{\theta_0}\right) = \sqrt{n}[I(\boldsymbol{\theta_0})]^{\frac{1}{2}}[\nabla q(\boldsymbol{\theta_0})]^{-1}\left(q\big(\boldsymbol{\hat{\theta}_n(X)}\big)-q(\boldsymbol{\theta_0})\right)\\
& - \frac{\sqrt{n}}{2}[I(\boldsymbol{\theta_0})]^{\frac{1}{2}}[\nabla q(\boldsymbol{\theta_0})]^{-1}\!\sum_{k=1}^{d}\sum_{l=1}^{d}\!\left(\hat{\theta}_n(\boldsymbol{x})_k - \theta_{0_k}\right)\!\left(\hat{\theta}_n(\boldsymbol{x})_l - \theta_{0_l}\right)\frac{\partial^2}{\partial\theta_k\partial\theta_l}q(\boldsymbol{\theta_0^*}).
\end{align}
Subtracting the quantity $\sqrt{n}K^{-\frac{1}{2}}\left(q\big(\boldsymbol{\hat{\theta}_n(X)}\big) - q\left(\boldsymbol{\theta_0}\right)\right)$ from both sides of \eqref{mid_step_proof} now gives
\begin{align}
\nonumber & \sqrt{n}[I(\boldsymbol{\theta_0})]^{\frac{1}{2}}\left(\boldsymbol{\hat{\theta}_n(X)} - \boldsymbol{\theta_0}\right) - \sqrt{n}K^{-\frac{1}{2}}\left(q\big(\boldsymbol{\hat{\theta}_n(X)}\big) - q(\boldsymbol{\theta_0})\right)\\
\nonumber & = \sqrt{n}\left([I(\boldsymbol{\theta_0})]^{\frac{1}{2}}[\nabla q(\boldsymbol{\theta_0})]^{-1} - K^{-\frac{1}{2}}\right)\left(q\big(\boldsymbol{\hat{\theta}_n(X)}\big)-q(\boldsymbol{\theta_0})\right)\\
\nonumber & \quad - \!\frac{\sqrt{n}}{2}[I(\boldsymbol{\theta_0})]^{\frac{1}{2}}[\nabla q(\boldsymbol{\theta_0})]^{-1}\!\sum_{k=1}^{d}\sum_{l=1}^{d}\!\left(\hat{\theta}_n(\boldsymbol{x})_k - \theta_{0_k}\right)\!\left(\hat{\theta}_n(\boldsymbol{x})_l - \theta_{0_l}\right)\!\frac{\partial^2}{\partial\theta_k\partial\theta_l}q(\boldsymbol{\theta_0^*}).
\end{align}
Componentwise we have that
\begin{align}
\label{final_step_rem_term}
\nonumber & \sqrt{n}\left[\left[I(\boldsymbol{\theta_0})\right]^{\frac{1}{2}}\right]_{[j]}\left(\boldsymbol{\hat{\theta}_n(X)} - \boldsymbol{\theta_0}\right) - \sqrt{n}\big[K^{-\frac{1}{2}}\big]_{[j]}\left(q\big(\boldsymbol{\hat{\theta}_n(X)}\big) - q(\boldsymbol{\theta_0})\right)\\
\nonumber & = \sqrt{n}\left(\left[\left[I(\boldsymbol{\theta_0})\right]^{\frac{1}{2}}\left[\nabla q(\boldsymbol{\theta_0})\right]^{-1}\right]_{[j]} - \big[K^{-\frac{1}{2}}\big]_{[j]}\right)\left(q\big(\boldsymbol{\hat{\theta}_n(X)}\big)-q(\boldsymbol{\theta_0})\right)\\
\nonumber & \quad - \frac{\sqrt{n}}{2}\left[\left[I(\boldsymbol{\theta_0})\right]^{\frac{1}{2}}\left[\nabla q(\boldsymbol{\theta_0})\right]^{-1}\right]_{[j]}\nonumber\\
\nonumber &\quad\times \sum_{k=1}^{d}\sum_{l=1}^{d}\left(\hat{\theta}_n(\boldsymbol{x})_k - \theta_{0_k}\right)\left(\hat{\theta}_n(\boldsymbol{x})_l - \theta_{0_l}\right)\frac{\partial^2}{\partial\theta_k\partial\theta_l}q(\boldsymbol{\theta_0^*})\\
\nonumber & = \sqrt{n}\sum_{k=1}^{d}\left\lbrace\left(\left[I(\boldsymbol{\theta_0})\right]^{\frac{1}{2}}\left[\nabla q(\boldsymbol{\theta_0})\right]^{-1} - K^{-\frac{1}{2}}\right)_{jk}\left(q_k\big(\boldsymbol{\hat{\theta}_n(X)}\big)-q_k(\boldsymbol{\theta_0})\right)\right\rbrace\\
\nonumber & \quad - \frac{\sqrt{n}}{2}\sum_{k=1}^{d}\left\lbrace\vphantom{(\left(\sup_{\theta:|\theta-\theta_0|\leq\epsilon}\left|l^{(3)}(\theta;\boldsymbol{X})\right|\right)^2}\left(\left[I(\boldsymbol{\theta_0})\right]^{\frac{1}{2}}\left[\nabla q(\boldsymbol{\theta_0})\right]^{-1}\right)_{jk}\right.\\
& \quad\left.\times \sum_{m=1}^{d}\sum_{l=1}^{d}\left(\hat{\theta}_n(\boldsymbol{x})_m - \theta_{0_m}\right)\left(\hat{\theta}_n(\boldsymbol{x})_l - \theta_{0_l}\right)\frac{\partial^2}{\partial\theta_m\partial\theta_l}q_k(\boldsymbol{\theta_0^*})\vphantom{(\left(\sup_{\theta:|\theta-\theta_0|\leq\epsilon}\left|l^{(3)}(\theta;\boldsymbol{X})\right|\right)^2}\right\rbrace.
\end{align}
We need to take into account that $\big|\frac{\partial^2}{\partial\theta_k\partial\theta_l}q_j(\boldsymbol{\theta})\big|$ is not uniformly bounded in $\boldsymbol{\theta}$. However, due to (R.C.3), it is assumed that for any $\boldsymbol{\theta_0} \in \Theta$ there exists $0< \epsilon = \epsilon(\boldsymbol{\theta_0})$ and constants $M_{jkl}(\boldsymbol{X}), \forall j,k,l \in \left\lbrace 1,2,\ldots,d \right\rbrace$, such that $\big|\frac{\partial^2}{\partial\theta_k\partial\theta_l}q_j(\boldsymbol{\theta})\big|\leq M_{jkl}(\boldsymbol{X})$ for all $\boldsymbol{\theta} \in \Theta$ such that $|\theta_j - \theta_{0_j}|<\epsilon, \; \forall j \in \left\lbrace 1,2,\ldots,d \right\rbrace$. Therefore, with $\epsilon > 0$ and $Q_{(m)}$ as in (\ref{cm}), the law of total expectation yields
\begin{align}
\label{lawtotalexp}
\nonumber R_2 = |\mathbb{E}[C_1]|  \leq \mathbb{E}|C_1| &= \mathbb{E}\left[|C_1|\middle| \left|Q_{(m)}\right|\geq\epsilon\right]\Prob\left(\left|Q_{(m)}\right|\geq\epsilon\right)\\
& \quad+ \mathbb{E}\left[|C_1|\middle| \left|Q_{(m)}\right| < \epsilon\right]\Prob\left(\left|Q_{(m)}\right| < \epsilon\right).
\end{align}
The Markov inequality applied to $\Prob\left(\left|Q_{(m)}\right|\geq\epsilon\right)$ yields
\begin{align}
\label{Markov}
\nonumber \Prob\left(\left|Q_{(m)}\right|\geq\epsilon\right) &= \Prob\left(\left|\hat{\theta}_n(\boldsymbol{X})_{(m)} - \theta_{0_{(m)}}\right|\geq\epsilon\right)\\
 &\leq \Prob\left(\sum_{j=1}^{d}\left(\hat{\theta}_n(\boldsymbol{X})_j - \theta_{0_j}\right)^2 \geq \epsilon^2 \right)\nonumber\\
&\leq \frac{1}{\epsilon^2}\mathbb{E}\left[\sum_{j=1}^{d}\left(\hat{\theta}_n(\boldsymbol{X})_j - \theta_{0_j}\right)^2\right].
\end{align}
In addition, $\Prob(|Q_{(m)}| < \epsilon) \leq 1$, which is a reasonable bound because the consistency of the MLE, a result of (R.C.1)-(R.C.5), ensures that $\big|\hat{\theta}_n(\boldsymbol{x})_{(m)} - \theta_{0_{(m)}}\big|$ should be small and therefore $\Prob(|Q_{(m)}| < \epsilon) = \Prob(|\hat{\theta}_n(\boldsymbol{x})_{(m)} - \theta_{0_{(m)}}| < \epsilon) \approx 1$. This result along with \eqref{Markov} applied to \eqref{lawtotalexp} gives
\begin{equation}
R_2 \leq 2 \frac{\|h\|}{\epsilon^2}\mathbb{E}\left[\sum_{j=1}^{d}\left(\hat{\theta}_n(\boldsymbol{X})_j - \theta_{0_j}\right)^2\right] + \mathbb{E}\left[|C_1|\middle| \left|Q_{(m)}\right| < \epsilon\right].
\end{equation}
For an upper bound on $\mathbb{E}\left[|C_1|\middle| \left|Q_{(m)}\right| < \epsilon\right]$, using \eqref{boundC1} and \eqref{final_step_rem_term} yields 
\begin{align*}
 R_2 &\leq 2\frac{\|h\|}{\epsilon^2}\mathbb{E}\left[\sum_{j=1}^{d}\!\left(\hat{\theta}_n(\boldsymbol{X})_j-\theta_{0_j}\right)^2\right]\!+\! \frac{|h|_1}{\sqrt{n}}\sum_{j=1}^{d}\!\Big\{\mathbb{E}\big[\big| A_{1j}(\boldsymbol{\theta_0},\boldsymbol{X})|\,\big|\,\left|Q_{(m)}\right|< \epsilon\big] \\
 &\quad+\mathbb{E}\big[\big| A_{2j}(\boldsymbol{\theta_0},\boldsymbol{X})|\,\big|\,\left|Q_{(m)}\right|< \epsilon\big]\Big\},
%\nonumber & \leq 2\frac{\|h\|}{\epsilon^2}\mathbb{E}\left(\sum_{j=1}^{d}\left(\hat{\theta}_n(\boldsymbol{X})_j-\theta_{0_j}\right)^2\right)\\
%\nonumber & \quad + \sqrt{n}\|h\|_1\sum_{j=1}^{d}\mathbb{E}\left[\left|A_{1j}(\boldsymbol{\theta_0},\boldsymbol{X})\right| + \left|A_{2j}(\boldsymbol{\theta_0},\boldsymbol{X})\right|\middle |\left|Q_{(m)}\right| < \epsilon\right]\\
%\nonumber & \leq 2\frac{\|h\|}{\epsilon^2}\mathbb{E}\left(\sum_{j=1}^{d}\left(\hat{\theta}_n(\boldsymbol{X})_j-\theta_{0_j}\right)^2\right)\\
%\nonumber & + \sqrt{n}\|h\|_1\sum_{j=1}^{d}\mathbb{E}\left[\left|\sum_{k=1}^d\left(\left[I(\boldsymbol{\theta_0})\right]^{\frac{1}{2}}\left[\nabla q(\boldsymbol{\theta_0})\right]^{-1} - K^{-\frac{1}{2}}\right)_{jk}\left(q_k\big(\boldsymbol{\hat{\theta}_n(X)}\big)-q_k(\boldsymbol{\theta_0})\right)\right|\middle |\left|Q_{(m)}\right| < \epsilon\right]\\
%& + \frac{\sqrt{n}\|h\|_1}{2}\sum_{j=1}^{d}\mathbb{E}\left[\left|\sum_{k,m,l=1}^{d}\left(\left[I(\boldsymbol{\theta_0})\right]^{\frac{1}{2}}\left[\nabla q(\boldsymbol{\theta_0})\right]^{-1}\right)_{jk}C_{kml}\left(\hat{\theta}_n(\boldsymbol{X})_m - \theta_{0_m}\right)\left(\hat{\theta}_n(\boldsymbol{X})_l - \theta_{0_l}\right)\right|\middle |\left|Q_{(m)}\right| < \epsilon\right]
\end{align*}
with $A_{1j}(\boldsymbol{\theta_0},\boldsymbol{X})$ and $A_{2j}(\boldsymbol{\theta_0},\boldsymbol{X})$ given as in the statement of the theorem. Summing our bounds for $R_1$ and $R_2$ gives the assertion of the theorem.
\hfill $\square$

\section{Normally distributed random variables}
\label{sec:normal_example}

%It is important to give an illustration that indicates the power of our general results through an example. We apply the results of Section \ref{sec:general_bound} in the case of $X_1,X_2, \ldots, X_n$ independent and identically distributed random variables from ${\rm N}(\mu,\sigma^2)$ with $\boldsymbol{\theta_0}=(\mu,\sigma^2)$. 

In this section, we illustrate the application of the general bound of Theorem \ref{Theorem_delta_multi} by considering the important case that $X_1,X_2, \ldots, X_n$ are i.i.d$.$ ${\rm N}(\mu,\sigma^2)$ random variables with $\boldsymbol{\theta_0}=(\mu,\sigma^2)$.  Here the density function is
\begin{equation}
\label{likelihoodexample}
f(x|\boldsymbol{\theta}) = \frac{1}{\sqrt{2\pi\sigma^2}}\exp\left\lbrace -\frac{1}{2\sigma^2}(x-\mu)^2 \right\rbrace,\quad x \in \mathbb{R}. 
\end{equation}
It is well-known that in this case the MLE exists, is unique and equal to $\boldsymbol{\hat{\theta}_n(X)} = \big(\bar{X}, \frac{1}{n}\sum_{i=1}^{n}\left(X_i-\bar{X}\right)^2\big)^{\intercal}$; see for example \cite[p.118]{Davison}.  As we shall see in the proof of the following corollary, the MLE has a representation of the form (\ref{qproperty}).  Thus, Theorem \ref{Theorem_delta_multi} can be applied to derive the following bound. Empirical results are given after the proof. 
\begin{corollary}\label{cor3.1}
Let $X_1, X_2, \ldots, X_n$ be i.i.d$.$  ${\rm N}\left(\mu, \sigma^2\right)$ random variables and let $\boldsymbol{Z} \sim \mathrm{N}_{2}(\boldsymbol{0}, I_{2\times 2})$.  Then, for all $h\in C_b^2(\mathbb{R}^2)$,
\begin{align}
\label{generalboundDeltamultivariate}
& \left|\mathbb{E}\left[h\left(\sqrt{n}\left[I(\boldsymbol{\theta_0})\right]^{\frac{1}{2}}\left(\boldsymbol{\hat{\theta}_n(X)} - \boldsymbol{\theta_0}\right)\right)\right]- \mathbb{E}[h(\boldsymbol{Z})]\right| \leq \frac{7|h|_2}{\sqrt{n}} + \frac{|h|_1}{\sqrt{2n}}.
\end{align}
\end{corollary}

\begin{remark}\label{remark3.2} An upper bound for the distributional distance between the MLE of the normal distribution, treated under canonical parametrisation, with unknown mean and variance has also been obtained in \cite{a15}. That bound is also of order $O(n^{-1/2})$, but our numerical constants are two orders of magnitude smaller. Furthermore, our test functions $h$ are less restrictive, coming from the class $C_b^2(\mathbb{R}^2)$ rather than the class $C_b^3(\mathbb{R}^2)$. Finally, the bound in \eqref{generalboundDeltamultivariate} does not depend on the parameter $\boldsymbol{\theta_0}=(\mu,\sigma^2)$, while the bound given in \cite{a15} depends on the natural parameters and blows up when the variance tends to zero or infinity.\end{remark}

\begin{proof}
From the general approach in the previous section, we want to have $q:\Theta \rightarrow \mathbb{R}^2$, such that
\begin{equation}
\nonumber q\big(\boldsymbol{\hat{\theta}_n(x)}\big) = \frac{1}{n}\sum_{i=1}^{n}\begin{pmatrix}
g_{1}(x_i)\\g_{2}(x_i)
\end{pmatrix}.
\end{equation}
We take
\begin{equation}
\label{qtheta1}
q(\boldsymbol{\theta}) = q(\theta_1,\theta_2) = \big(\theta_1, \theta_2 + (\theta_1-\mu)^2\big)^\intercal,
\end{equation}
which then yields
\begin{align}
\label{qfunctionexample}
\nonumber q(\boldsymbol{\hat{\theta}_n(X)}) &= q\bigg(\bar{X}, \frac{1}{n}\sum_{i=1}^{n}(X_i - \bar{X})^2\bigg)^\intercal\! =\! \bigg(\bar{X}, \frac{1}{n}\sum_{i=1}^{n}(X_i - \bar{X})^2 + \left(\bar{X} - \mu\right)^2\bigg)^\intercal\\
& = \bigg(\frac{1}{n}\sum_{i=1}^{n}X_i, \frac{1}{n}\sum_{i=1}^{n}(X_i - \mu)^2\bigg)^\intercal,
\end{align}
and therefore
\begin{equation}
\label{gfunctionexample}
g_{1}(x_i) = x_i, \quad g_{2}(x_i) = \left(x_i - \mu\right)^2,
\end{equation}
and thus the MLE has a representation of the form (\ref{qproperty}).

We now note that the Jacobian matrix of $q(\boldsymbol{\theta})$ evaluated at $\boldsymbol{\theta_0} = (\mu, \sigma^2)$ is 
\begin{equation}
\label{qnablafunction_example}
\nabla q(\boldsymbol{\theta_0}) = \begin{pmatrix}
1 & 0\\ 0 & 1
\end{pmatrix}.
\end{equation}
Furthermore, it is easily deduced from \eqref{qtheta1} that $\forall k,m,l \in \left\lbrace 1,2\right\rbrace$, we have that  $\big|\frac{\partial^2}{\partial\theta_m\partial\theta_l}q_k(\boldsymbol{\theta})\big| \leq 2$. Hence, using Remark \ref{remarkdeltamulti} we conclude that the general bound is as in \eqref{uniformlyboundedq} and no positive constant $\epsilon$ is necessary. Now, for $l(\boldsymbol{\theta_0};\boldsymbol{X})$ denoting the log-likelihood function, we obtain from \eqref{likelihoodexample} that the expected Fisher information matrix is 
\begin{equation}
\label{fisher_example}
I(\boldsymbol{\theta_0}) = -\frac{1}{n}\mathbb{E}\left[\nabla^{\intercal}\nabla \left(l(\boldsymbol{\theta_0};\boldsymbol{X})\right) \right] = \begin{pmatrix}
\frac{1}{\sigma^2} & 0\\ 0 & \frac{1}{2\sigma^4}
\end{pmatrix}.
\end{equation}

The results in \eqref{qnablafunction_example} and \eqref{fisher_example} yield
\begin{equation}
\label{Kmatrix_example}
K = \left[\nabla q(\boldsymbol{\theta_0})\right]\left[I(\boldsymbol{\theta_0})\right]^{-1}\left[\nabla q(\boldsymbol{\theta_0})\right]^{\intercal} = [I(\boldsymbol{\theta_0})]^{-1} = \begin{pmatrix}
\sigma^2 & 0\\0&2\sigma^4
\end{pmatrix}.
\end{equation}
Using \eqref{W}, \eqref{qfunctionexample}, \eqref{gfunctionexample} and \eqref{Kmatrix_example}, we get that 
\begin{align}
\nonumber \boldsymbol{W} &= \frac{K^{-\frac{1}{2}}}{\sqrt{n}}\sum_{i=1}^{n}\left[g(X_i) - q(\boldsymbol{\theta_0})\right] = \frac{1}{\sqrt{n}}\begin{pmatrix}
\frac{1}{\sigma}& 0\\0 & \frac{1}{\sqrt{2}\sigma^2}
\end{pmatrix}\sum_{i=1}^{n}\begin{pmatrix}
X_i - \mu\\ (X_i- \mu)^2 - \sigma^2
\end{pmatrix}\\
\nonumber & = \begin{pmatrix}
\sum_{i=1}^{n}\frac{X_i - \mu}{\sigma\sqrt{n}}\\ \sum_{i=1}^{n}\frac{\left(X_i - \mu\right)^2 - \sigma^2}{\sqrt{2n}\sigma^2}
\end{pmatrix},
\end{align}
and therefore, for $i\in\{1,\ldots,n\}$,
\begin{equation}
\nonumber \xi_{i1} = \frac{X_i - \mu}{\sigma}, \qquad \xi_{i2} = \frac{(X_i - \mu)^2 - \sigma^2}{\sqrt{2}\sigma^2}.
\end{equation}

For the first term of the general upper bound in \eqref{uniformlyboundedq}, we are required to bound certain expectations involving $\xi_{ij}$. We first note that $\frac{X_i-\mu}{\sigma}\stackrel{\mathrm{d}}{=}Z\sim \mathrm{N}(0,1)$. Using the standard formulas for the moments and absolute moments of the standard normal distribution (see \cite{w12}), we have, for all $i=1,\ldots,n$,
\begin{equation*}
\mathbb{E}|\xi_{i1}|=\mathbb{E}|Z|=\sqrt{\frac{2}{\pi}}, \quad \mathbb{E}\xi_{i1}^2=\mathbb{E}Z^2=1, \quad \mathbb{E}|\xi_{i1}|^3=\mathbb{E}|Z|^3=2\sqrt{\frac{2}{\pi}}
\end{equation*}
and, by H\"{o}lder's inequality,
\begin{eqnarray*}
\mathbb{E}|\xi_{i2}|&=&\frac{1}{\sqrt{2}}\mathbb{E}\left|Z^2-1\right|\leq\frac{1}{\sqrt{2}}(\mathbb{E}\left(Z^2-1\right)^2)^{1/2}=1, \\
 \mathbb{E}\xi_{i2}^2&=&\frac{1}{2}\mathbb{E}\left(Z^2-1\right)^2=1, \\
\mathbb{E}\left|\xi_{i2}\right|^3 & =& \frac{1}{2^{3/2}}\mathbb{E}\left|Z^2-1\right|^3 \leq \frac{1}{2^{3/2}}\left(\mathbb{E}\left(Z^2-1\right)^4\right)^{3/4}\\
&=&\frac{1}{2^{3/2}}\left(\mathbb{E}Z^8-4\mathbb{E}Z^6+6\mathbb{E}Z^4-4\mathbb{E}Z^2+1\right)^{3/4} \\
&=&\frac{1}{2^{3/2}}(105-60+18-4+1)^{3/4} =15^{3/4}.
\end{eqnarray*}
Therefore,
\begin{align}
\label{term1}
\nonumber \sum_{i=1}^n\sum_{j,k,l}^2\mathbb{E}|\xi_{ij}\xi_{ik}\xi_{il}|&=\sum_{i=1}^n(\mathbb{E}|\xi_{i1}|^3+3\mathbb{E}\xi_{i1}^2\mathbb{E}|\xi_{i2}|+3\mathbb{E}|\xi_{i1}|\mathbb{E}\xi_{i2}^2+\mathbb{E}|\xi_{i2}|^3) \\
&\leq\sum_{i=1}^n\bigg(2\sqrt{\frac{2}{\pi}}+3+3\sqrt{\frac{2}{\pi}}+15^{3/4}\bigg)<14.612n,
\end{align}
and
\begin{equation}
\label{term2}
\sum_{i=1}^n\sum_{j,k,l=1}^2|\mathbb{E}\xi_{ij}\xi_{ik}|\mathbb{E}|\xi_{il}|=\sum_{i=1}^n\sum_{j,l=1}^2\mathbb{E}\xi_{ij}^2\mathbb{E}|\xi_{il}|=2n\bigg(1+\sqrt{\frac{2}{\pi}}\bigg).
\end{equation}
Applying the results of \eqref{term1} and \eqref{term2} to the first term of the general bound in \eqref{uniformlyboundedq} yields
\begin{align}
\label{firsttermexample}
\nonumber &\frac{\sqrt{\pi}|h|_2}{4\sqrt{2}n^{3/2}}\sum_{i=1}^n\sum_{j,k,l=1}^2\Big\{\mathbb{E}|\xi_{ij}\xi_{ik}\xi_{il}|+2|\mathbb{E}\xi_{ij}\xi_{ik}|\mathbb{E}|\xi_{il}|\Big\}\\
 &\quad \leq\frac{\sqrt{\pi}|h|_2}{4\sqrt{2}n^{3/2}}\bigg(4n\bigg(1+\sqrt{\frac{2}{\pi}}\bigg)+14.612n\bigg)= \frac{6.833|h|_2}{\sqrt{n}} \leq \frac{7|h|_2}{\sqrt{n}}.
\end{align}

For the second term, we need to calculate $A_{1j}(\boldsymbol{\theta_0},\boldsymbol{X})$ and $A_{2j}(\boldsymbol{\theta_0},\boldsymbol{X})$ as in Theorem \ref{Theorem_delta_multi}. The results of \eqref{qnablafunction_example} and \eqref{Kmatrix_example} yield to
\begin{equation}
\nonumber \left[I(\boldsymbol{\theta_0})\right]^{\frac{1}{2}}\left[\nabla(q(\boldsymbol{\theta_0}))\right]^{-1} - K^{-\frac{1}{2}} = \boldsymbol{0}_{2\times 2},
\end{equation}
where $\boldsymbol{0}_{2\times 2}$ denotes the 2 by 2 zero matrix. Therefore $A_{1j}(\boldsymbol{\theta_0},\boldsymbol{X})=0$ for $j=1,2$. In order to calculate $A_{2j}(\boldsymbol{\theta_0},\boldsymbol{X})$, we note that
\begin{align}
\nonumber & \frac{\partial^2}{\partial\theta_m\partial\theta_l}q_1(\boldsymbol{\theta}) = 0, \quad\forall m,l\in \left\lbrace 1,2 \right\rbrace,\\
\nonumber & \frac{\partial^2}{\partial\theta_1\partial\theta_2}q_2(\boldsymbol{\theta}) = \frac{\partial^2}{\partial\theta_2\partial\theta_1}q_2(\boldsymbol{\theta}) = \frac{\partial^2}{\partial\theta_2^2}q_2(\boldsymbol{\theta}) = 0,\\
\nonumber & \frac{\partial^2}{\partial\theta_1^2}q_2(\boldsymbol{\theta}) = 2.
\end{align}
Using this result as well as \eqref{qnablafunction_example} and \eqref{fisher_example}, we get that
\begin{align*}
 A_{21}(\boldsymbol{\theta_0},\boldsymbol{X}) &= 0,\\
A_{22}(\boldsymbol{\theta_0},\boldsymbol{X}) &=  \frac{n}{2\sqrt{2}\sigma^2}2\left(\bar{X} - \mu\right)^2 = \frac{n}{\sqrt{2}\sigma^2}\left(\bar{X} - \mu\right)^2.
\end{align*}
Therefore the last term of the general upper bound in \eqref{uniformlyboundedq} becomes
\begin{align}
\label{A2example}
\nonumber  &\frac{|h|_1}{\sqrt{n}}\sum_{j=1}^2\big\{\mathbb{E}|A_{1j}(\boldsymbol{\theta_0},\boldsymbol{X})| +\mathbb{E}|A_{2j}(\boldsymbol{\theta_0},\boldsymbol{X})|\big\}\\
  &\quad= \frac{|h|_1}{\sqrt{n}}\mathbb{E}\left| A_{22}(\boldsymbol{\theta_0},\boldsymbol{X}) \right| = \frac{\sqrt{n}}{\sqrt{2}\sigma^2}|h|_1\mathbb{E}\left[\left(\bar{X}- \mu\right)^2\right] = \frac{|h|_1}{\sqrt{2n}}.
\end{align}
Finally, the bounds in \eqref{firsttermexample} and \eqref{A2example} are applied to \eqref{uniformlyboundedq} to obtain the assertion of the corollary.
\end{proof}
%{\raggedright{\textbf{Empirical results}}}
%\vspace{0.1in}
%\\
Here, we carry out a large-scale simulation study to investigate the accuracy of the bound in \eqref{generalboundDeltamultivariate}. For $n=10^j, j=3,4,5,6$, we start by generating $10^4$ trials of $n$ random independent observations, $X$, following the ${\rm N}(\mu, \sigma^2)$ distribution. We take $\mu =1$, $\sigma^2 = 1$ for our simulations. Then $\sqrt{n}\left[I(\boldsymbol{\theta_0})\right]^{\frac{1}{2}}\big(\boldsymbol{\hat{\theta}_n(X)} - \boldsymbol{\theta_0}\big)$ is evaluated in each trial, which in turn gives a vector of $10^4$ values. The function $h(x,y) = \left(x^2+y^2+1\right)^{-1}$, which belongs in $C_b^2(\mathbb{R}^2)$, is then applied to these values in order to get the sample mean, which we denote by  $\hat{\EE}\big[h\big(\sqrt{n}\left[I(\boldsymbol{\theta_0})\right]^{\frac{1}{2}}\big(\boldsymbol{\hat{\theta}_n(X)} - \boldsymbol{\theta_0}\big)\big)\big]$. For the function $h$, it is straightforward to see that
\begin{equation*}
\label{boundsh}
|h|_1 = \frac{3\sqrt{3}}{8},\quad |h|_2 = 2.
\end{equation*}
We use these values to calculate the bound in \eqref{generalboundDeltamultivariate}. We define $$Q_{h}(\boldsymbol{\theta_0}):=\left|\hat{\EE}\left[h\left(\sqrt{n}\left[I(\boldsymbol{\theta_0})\right]^{\frac{1}{2}}\left(\boldsymbol{\hat{\theta}_n(X)} - \boldsymbol{\theta_0}\right)\right)\right] - \tilde{\EE}[h(\boldsymbol{Z})]\right|,$$
where $\tilde{\EE}[h(\boldsymbol{Z})] = 0.461$ is the approximation of $\EE[h(\boldsymbol{Z})]$ up to three decimal places. We compare $Q_{h}(\boldsymbol{\theta_0})$ with the bound in \eqref{generalboundDeltamultivariate}, using the difference between their values as a measure of the error. The results from the simulations are shown in Table \ref{tableresultsmultinormal} below.
\begin{table}[H]
\caption{Simulation results for the ${\rm N}(1,1)$ distribution}
\vspace{0.04in}
\centering
\begin{tabular}{l|l|l|l}
$n$ & $Q_h(\boldsymbol{\theta_0})$ & Upper bound & Error\\
\hline
\hline
$10^3$ & 0.011 & 0.457 & 0.446\\
\hline
$10^4$ & 0.010 & 0.145 & 0.135\\
\hline
$10^5$ & 0.009 & 0.046  & 0.037\\
\hline
$10^6$ & 0.006 & 0.014 & 0.008
  \end{tabular}
\label{tableresultsmultinormal}
  \end{table}
We see that the bound and the error decrease as the sample size gets larger. To be more precise, when at each step we increase the sample size by a factor of ten, the value of the upper bound drops by a factor close to $\sqrt{10}$, which is expected since the order of the bound is $O(n^{-1/2})$, as can be seen from \eqref{generalboundDeltamultivariate}.

\section*{Acknowledgements} This research occurred whilst AA was studying for a DPhil at the University of Oxford, supported by a Teaching Assistantship Bursary from the Department of Statistics, University of Oxford, and the EPSRC grant EP/K503113/1. RG acknowledges support from EPSRC grant EP/K032402/1 and is currently supported by a Dame Kathleen Ollerenshaw Research Fellowship. We would like to thank the referee for their helpful comments and suggestions.

\end{document}